\documentclass[12pt]{amsart}
\usepackage{amssymb,latexsym,amsmath,amsthm,amscd}
\usepackage{setspace}
\usepackage[all]{xy} \xyoption{arc}
\usepackage[left=3cm,top=2cm,right=3cm,bottom = 2cm]{geometry}
\usepackage{graphicx}
\usepackage[usenames,dvipsnames]{color}
\usepackage{charter}

\theoremstyle{plain}
\newtheorem{thm}{Theorem}
\newtheorem{lem}[thm]{Lemma}

\newtheorem{cor}[thm]{Corollary}

\theoremstyle{definition}
\newtheorem{dfn}[thm]{Definition}

\theoremstyle{remark}
\newtheorem{rmk}[thm]{Remark}

\DeclareMathOperator{\Tr}{Tr}

\DeclareMathOperator{\im}{im}

\DeclareMathOperator{\Res}{Res}

\newcommand*{\df}{\mathrel{\vcenter{\baselineskip0.5ex \lineskiplimit0pt
                     \hbox{\scriptsize.}\hbox{\scriptsize.}}} =}

\providecommand{\abs}[1]{\left\lvert#1\right\rvert}

\newcommand{\QQ}{\mathbf{Q}}

\newcommand{\ZZ}{\mathbf{Z}}

\begin{document}

 \title[$p$-adic vertex algebras]{Invitation to $p$-adic Vertex Algebras}
\title{Invitation to $p$-adic Vertex Algebras}

\author{Cameron Franc}
\address{Department of Mathematics, McMasters University, Hamilton, Ca}
\curraddr{}
\email{cameron.franc@gmail.com}
\thanks{}

\author{Geoffrey Mason}
\address{Department of Mathematics, University of California, Santa Cruz}
\curraddr{}
\email{gem@ucsc.edu}
\thanks{We thank the Simons Foundation, grant $\#427007$, for its support, and NSERC for support via a Discovery grant.}


\date{}

\begin{abstract} An overview of the authors' ideas about the process of completing a $p$-adically normed space in the setting of vertex operator algebras.\ We focus in particular on the $p$-adic Heisenberg VOA and its connections with $p$-adic modular forms.
\end{abstract}

\maketitle

\tableofcontents
 
  \section{Introduction}
  
The idea of completion is a basic technique in mathematics.\ It can be carried out in a topological or algebraic setting, and even categorically.\ In the context of completing rational vector spaces with respect to a $p$-adic norm, this process is referred to as \emph{localizing}. The purpose of the present paper is to explain the process of completion (based on an ultrametric norm) as it pertains to vertex algebras, following two papers of the authors \cite{FM1}, \cite{FM2}.\ It is an expanded version of a short talk on this subject presented by the second author at the joint AMS/UMI meeting in Palermo.\ Before going into details we will discuss \textit{why} one might want a local theory of vertex algebras.

\medskip
The mathematical theory of vertex algebras has thus far been developed almost exclusively over the field of complex numbers.\ Yet according to \cite{DKKVZ}, $p$-adic mathematical physics dates back to 1987.\
 In current $p$-adic quantum mechanics   the state space
is nonarchimedean  but the fields are invariably complex-valued; $p$-adic fields and  $p$-adic locality appear to be absent, although there have been calls for such a theory,
e.g., \cite{Mar}.\ Our $p$-adic vertex operator algebras fill this lacuna.\ They amount to the axiomatization of a chiral half of a nonArchimedean 2-dimensional bosonic CFT in which the traditional Hilbert space is replaced by a $p$-adic Banach space.

\medskip
Mathematically, one of our main motivations comes from \textit{the unreasonable effectiveness of modular forms in VOA theory.}\ Taking this maxim seriously and locally led us  to attempt to incorporate \textit{$p$-adic modular forms} into the theory of vertex algebras.\ 
There are two rather different approaches to $p$-adic modular forms.\ Serre originally introduced these objects \cite{S1} and his approach
is elementary in the sense that it requires  only a relatively modest background.\ Serre's work was subsequently generalized by Katz \cite{Ka}.\ This approach is based on the geometric
perspective of modular forms as sections of bundles over the moduli space of elliptic curves.\ We shall only employ Serre's $p$-adic modular forms here, although we are in no doubt that
a geometric approach will be necessary to properly understand $p$-adic vertex algebras.

\medskip
Apart from $p$-adic modularity, there are potential applications of $p$-adic vertex algebras to the  study of  \textit{genus theory} for  vertex algebras.\ The general idea is based on the well-known analogy between VOA theory and  lattice theory.\ Thus one may anticipate the emergence of local-global principles for vertex algebras in which  $p$-adic vertex algebras participate.\
 For further background on this new subject, see \cite{Mor} and M\"{o}ller's lecture at this conference \cite{Moll}.
 
  \medskip
  Now let us return to the contents of the present paper.\ Our goal is to explain the results in
  \cite{FM1}, \cite{FM2} although we shall here adopt a rather different perspective.\ In \textit{loc.\ cit.}\ our viewpoint was essentially number-theoretic and to explain this let $S_{alg}$ denote the standard Heisenberg vertex operator algebra of central charge $1$ defined over the base field $\mathbf{Q}_p$ (the idiosyncratic notation is justified below.)\ It is known  that the $1$-point function (actually, a renormalized version 
  of this) defines a linear surjection
  \begin{eqnarray}\label{fdfn}
  S_{alg}\stackrel{f}{\longrightarrow} Q
  \end{eqnarray}
  onto the space $Q:=\mathbf{Q}_p[E_2, E_4, E_6]$ of quasimodular forms.\ One way to define the space of $p$-adic modular forms $M_p$
  \textit{a la} Serre  is as the completion of $Q$
  with respect to a certain ultrametric norm and in this instance at least, the 'incorporation of $p$-adic modular forms into vertex algebra theory'
  refers to the completion of (\ref{fdfn}) to a commuting diagram of continuous maps
  \begin{eqnarray}\label{commdiag}
  \xymatrix{ S_1\ar[r]^{\hat{f}}& M_p\\
  S_{alg}\ar[u]\ar[r]_f& Q\ar[u]
  }
  \end{eqnarray}
  
  We will explain how to make sense of this.\ To be useful, $S_{alg}$ needs to be equipped with a 'compatible' $p$-adic norm, and the completion $S_1$
  of $S_{alg}$ is then a $p$-adic Banach space which should be some kind of vertex  object (it is our $p$-adic Heisenberg VOA in this particular case); $\hat{f}$ will then be the $p$-adic $1$-point function for $S_1$.\ A major number-theoretic question for us is whether or not 
  $\hat{f}$ is (like $f$) a surjection.\ This question remains open but it has guided our ideas from the very outset of this work.

  \medskip
  Setting aside $p$-adic modular forms for now, it is informative to work in a more categorical context, so we begin with a pair of categories $\mathbf{C}$ and $\mathbf{D}$.\ The objects of 
  $\mathbf{D}$ are \textit{$p$-adically normed VOAs}, i.e., classical VOAs defined over the field of $p$-adic numbers $\mathbf{Q}_p$ and equipped with a compatible $p$-adic norm.\ The main issue here is to decide what a $p$-adic field, or $p$-adic vertex operator, should be.\ The objects of $\mathbf{C}$ are our $p$-adic VOAs, the chiral half of a $2$-dimensional $p$-adic bosonic CFT. Their underlying Fock spaces are $p$-adic Banach spaces.
The process of completing the objects in 
$\mathbf{D}$ produces an equivalence of categories 
$\mathbf{C}\sim\mathbf{D}$.\
The completion process imposes on the completed space a $p$-adic version of the usual Jacobi identity which, to an untrained eye, looks identical
to the standard Jacobi identity.\ This is very beneficial inasmuch as we may apply techniques and results from classical vertex algebra theory
  to $p$-adic VOAs.

  \medskip
  Having pointed out the similarities between vertex algebras and their $p$-adic counterparts, we should emphasize that there are also significant differences.\ Such differences mainly pertain to the  classical VOA axioms concerning the spectral decomposition into $L(0)$-eigenspaces.\ These typically  \textit{cannot} extend to the $p$-adic VOAs, because $p$-adic VOAs are generally of uncountable dimension over $\mathbf{Q}_p$, whereas of course classical VOAs have countable dimension.\ This issue, at least as far as the $p$-adic Heisenberg VOA $S_1$ is concerned, is related to  (\ref{commdiag}) and the connection with $p$-adic modular forms sheds light on the situation.\ Thus we discuss this phenomenon only for $S_1$, though much of our work applies more broadly.

  \medskip
  The previous remarks suggest a closer examination of the spectral properties of the bounded operator $L(0)$ in its action on $S_1$.\ It turns out that the spectrum of $L(0)$ is unchanged, i.e., the point spectrum is $\mathbf{Z}_{\geq0}$ and there is no continuous spectrum.\ 
  
  \medskip
  For reasons that we do not fully understand, the genus $1$
  operator $L[0]$ has a much bigger point spectrum than $L(0)$, in fact  every element in 
  $\mathbf{Z}_p$ is an eigenvalue.\ However this statement is complicated by the fact that we must treat $L[0]$ as an unbounded operator on $S_1$.\
  Nevertheless we can use Serre's theory of $p$-adic modular forms and their $p$-adic weights to make a preliminary analysis of the point spectrum of $L[0]$.\ It turns out that many states in
  $S_1$ naturally carry a $p$-adic weight that lies in so-called weight space $X$, which is the profinite group where the weights of $p$-adic modular forms reside.\ This new notion of $p$-adic weight of a state in $S_1$
  generalizes the usual conformal weight in a standard VOA and, at least for now, requires $p$-adic modularity to establish its existence.

  \medskip
 The subject of $p$-adic vertex algebras involves a pleasing mixture of ideas from
 $p$-adic functional analysis, $p$-adic modular forms, and of course vertex algebras.\ Aware that a reader may have little knowledge of any of these topics (even the third one)
 we present some background in each of them - at least enough so that an assiduous reader can follow the gist of the narrative.
  
  \medskip
  We are indebted to  Darlayne Addabbo, Lisa Carbone and Roberto Volpato for organizing the special session in Palermo and for affording us the opportunity to present our work.
  
  \section{Vertex algebras}\label{SVA}

In this Section we review the definition of a classical vertex algebra  over a unital, commutative ring $k$ \cite{Ma1} and some related aspects of this theory.\ Many readers will be familiar with this material and we present it here to establish notation and to facilitate a comparison with the $p$-adic version that comes later.\ Throughout, we
fix $k$ and a $k$-module $V$.\ Elements in $V$ are called \textit{states}.

\subsection{Fields on $V$}\label{SSfieldsonV}
We are concerned with  formal power series in a variable $z$ with coefficients lying in $End_k(V)$.\ They are conventionally written in the form
\begin{eqnarray}\label{dist}
a(z) = \sum_{n\in\mathbf{Z}} a(n)z^{-n-1} \in End_k(V)[[z, z^{-1}]].
\end{eqnarray}

\noindent
Here, $\{a(n)\}$ is any sequence of $k$-endomorphisms of $V$, called \textit{modes}, and the symbol $a$ is insignificant.\ A \textit{field} on $V$ is a series $a(z)$ with the property that
for all $b\in V$ we have $a(n)b=0$ for $n\gg0$.\ We set
\begin{eqnarray}\label{fd}
\mathcal{F}(V) := \{a(z) \mid \mbox{$a(z)$ is a field on $V$}\}.
\end{eqnarray}
Clearly, $\mathcal{F}(V)$ is a $k$-submodule of $End_k(V)[[z,z^{-1}]]$.\ It is natural to define
\begin{eqnarray*}
a(z)b:= \sum_{n\in\mathbf{Z}} a(n)bz^{-n-1}
\end{eqnarray*}
so that if $a(z)$ is a field then  $a(z)b\in V[z^{-1}][[z]]$.

\subsection{$Y$-maps}
A \textit{$Y$-map} is a morphism of $k$-modules 
\begin{eqnarray*}
Y: V\rightarrow \mathcal{F}(V),\ a \mapsto Y(a, z):=\sum_{n\in\mathbf{Z}}a(n)z^{-n-1}.
\end{eqnarray*}

Here the sequence of modes is associated with a specific state $a$.\ $Y(a, z)$ is the
\textit{vertex operator} associated to $a$.\ Then
\begin{eqnarray*}
Y(a, z)b=\sum_{n\in\mathbf{Z}}a(n)bz^{-n-1}
\end{eqnarray*}
and we may usefully consider $a(n)b$ as a bilinear product on $V$ for each $n$.\ In this way $V$ becomes a $k$-algebra with a countable infinity of 
bilinear products for which the 'singularities are at worst poles'.


\subsection{Vertex algebras over $k$}\label{SSVA}
A \textit{vertex algebra over $k$} is an algebra $V$ in the above sense equipped with an identity
(called the \textit{vacuum state}). Precisely, it is a triple $(V, Y, \mathbf{1})$ satisfying the following axioms:
\begin{eqnarray*}
&& (a)\ (\mbox{State-field correspondence})\ Y: V\rightarrow \mathcal{F}(V)\ \mbox{is a $Y$-map},\\
&& (b)\ (\mbox{Creativity})\ Y(a, z)\mathbf{1} = a+ O(z), \\
&& (c)\ (\mbox{Jacobi identity) for all}\ a, b, c\in V\ \mbox{and all}\  r, s, t\in \mathbf{Z}\ \mbox{we have}\\
&&\ \ \ \ \ \ \ \ \ \ \ \ \ \ \sum_{i=0}^{\infty} {r\choose i}(a(t+i)b)(r+s-i)c = \\
&&\sum_{i=0}^{\infty} (-1)^i{t\choose i}\left\{a(r+t-i)b(s+i)c- b(s+t-i)a(r+i)c  \right\}.
\end{eqnarray*}

Part (b) says that $a(-1)\mathbf{1}=a$ and $a(n)\mathbf{1}=0\ (n\geq0)$ and it can be deduced \cite{Ma1} that in fact
$Y(\mathbf{1}, z)=Id_V$.\ It is in this sense that the vacuum state $\mathbf{1}$ is a sort of identity.\ Note that the creativity axiom implies that
$Y$ injective.\ The name state-field correspondence comes from CFT.\ We point out  that  each of the sums in (c) contain only finitely many nonzero terms.\ This follows because
$Y(a, z)$ and $Y(b, z)$ both belong to $\mathcal{F}(V)$.\ 

\medskip
The crux of VA theory resides in the Jacobi identity.\ The formulation of (c) is 
opaque as it stands although it will be convenient for us when we introduce our $p$-adic Jacobi identity in Subsection \ref{SSpVA}.\ Several equivalent formulations
of (c) are available which the reader may find more enlightening (cf.\ \cite{Ma1}, for example) but we will not deal with them here.

\subsection{Vertex operator algebras}\label{SSVOA} We will only need to consider VOAs when $k$ has characteristic $0$.\
With this tacit assumption, we define a VOA over $k$ to be a quadruple $(V, Y, \mathbf{1}, \omega)$ such that $(V, Y, \mathbf{1})$ is a vertex algebra over $k$
and $\omega\in V$ (called the \textit{Virasoro} state) has the following additional properties:
\begin{eqnarray*}
&&(a)\  Y(\omega, z):=\sum_{n\in\mathbf{Z}} L(n)z^{-n-2},\ \mbox{where} \\
&&\ \ \ \  [L(m), L(n)]=(m-n)L(m+n)+ \delta_{m, -n} \tfrac{(m^3-m)}{12}cId_V\\
&&\ \ \ \  \mbox{for some $c\in k$ called the central charge of $V$}, \\
&&(b)\ (\mbox{translation covariance})\ [L((-1), Y(a,z)] = \partial_zY(a, z)\ \ (a\in V).
\end{eqnarray*}
Furthermore the operator $L(0)$ has the following spectral properties:
\begin{eqnarray*}
&&(c)\ V = \oplus_{\ell\in\mathbf{Z}} V_{\ell}\ \mbox{where $V_{\ell}$ is a free $k$-module of finite rank, $V_{\ell}=0$ for $\ell \ll 0$, and}\\
&&(d)\ V_{\ell}:=\{a\in V \mid L(0)a=\ell a\}.
\end{eqnarray*}

$L(0)$ is sometimes called the \textit{Hamiltonian} and we will adopt this language.\ Finally, in what follows we often refer to a VOA as defined above
as an \textit{algebraic} VOA in order to distinguish it from its $p$-adically completed counterpart.

\begin{rmk} In fact, part (b) can be deduced from the other axioms (for example, as in \cite{Ma1}) so it may be omitted from the definition.\ We have included it here to emphasize its importance.
\end{rmk}

\section{$p$-adic Banach spaces}
For background on $p$-adic functional analysis we refer the reader to \cite{BGR} and \cite{S}.\ 
In this paper, an \textit{ultrametric ring} is a commutative ring $R$ equipped with an absolute value $\abs{r}$ satisfying the usual axioms of an absolute value in addition to the ultrametric triangle inequality
\[
  \abs{r+s} \leq \sup\{\abs{r},\abs{s}\}.
\]
The example of interest here concerns the rational numbers endowed with the $p$-adic ultrametric defined as follows: factor $a/b \in \QQ^\times$ in the form $p^r a'/b'$ where $a',b' \in\ZZ$ are both coprime to $p$. Then set
\[
\abs{a/b} = \abs{a/b}_p = p^{-r}.
\]
As above, we will suppress the subscript $p$ in the notation for the $p$-adic absolute value.

\medskip
$\QQ_p$ is the completion of $\QQ$ for the $p$-adic absolute value, and $\ZZ_p$ is the closed unit ball inside of $\QQ_p$.\ It can be equivalently realized as the completion of $\ZZ$ with respect to the $p$-adic absolute value.\ Elements in $\QQ_p$ can be expressed as formal Laurent series in powers of $p$, with coefficients in some set of coset representatives for $\ZZ_p/p\ZZ_p \cong \ZZ/p\ZZ$, typically taken to be $0$, $1$, $\ldots$, $p-1$. Using the Laurent series representation, arithmetic is performed ``with carries'' in $\QQ_p$, so that the resulting arithmetic structure differs from that in a formal power series ring.

\medskip
There is a well-developed theory of ultrametric analysis that contains many well-known results from classical analysis, as well as suprising new phenomenon (loc.\ cit.).\ For us, one of the most useful features of an ultrametric is that a series
\[
\sum_{n\geq 0}a_n
\]
in an ultrametric ring converges if, and only if, 
$\lim a_n = 0$.

\medskip
We will also need to make use of normed modules over $R=\QQ_p$ and $R=\ZZ_p$. These are $R$-modules endowed with a compatible absolute value that also satisfies the ultrametric inequality. When the underlying module is complete with respect to the ultrametric, we shall call the module a $p$-adic Banach space. The Cauchy sequence construction for completions works as usual, so that one can complete normed modules over $R$ to obtain Banach spaces.

\medskip
Let $V$ be a normed space over $\QQ_p$.\ In our applications $V$ will generally be infinite dimensional, and so not all linear maps $T\colon V\to V$ need be bounded.\ Recall that  $T$  is bounded if the following supremum is finite:
\[
\sup_{v \in V\setminus \{0\}} \frac{\abs{T(v)}}{\abs{v}} < \infty.
\]
In this case we define $\abs{T}$ to be the value of this supremum, and call $T$ a \emph{bounded} linear operator on $V$.\ It is well-known by the general theory of normed linear spaces that \textit{boundedness is equivalent to continuity} in this setting.\ Let $B(V)$ be the space of bounded linear operators on $V$.\ 
$B(V)$ is itself a normed space with respect to the sup norm defined above, and if $V$ is a Banach space then so is $B(V)$.

\section{$p$-adic modular forms}\label{Spadic}
In this section we recall the rudiments of the $p$-adic theory of modular forms of tame level one, as introduced in
\cite{S1}.

\medskip
Recall that the ring of (rational) modular forms of level one is a graded ring $M \df \QQ[Q,R]$ where $Q$ and $R$ are generators of weights $4$ and $6$, respectively. These Eisenstein series can be described by their well-known Fourier expansions
\begin{align*}
Q &= 1 + 240\sum_{n\geq 1} \sigma_3(n)q^n,\\
R &= 1 -504 \sum_{n\geq 1} \sigma_5(n)q^n,
\end{align*}
where as usual $\sigma_k(n) = \sum_{d \mid n } d^k$ denotes the $k$th power divisor sum function, and $q=e^{2\pi i \tau}$. A slightly larger ring of (rational) quasimodular forms $M'=\QQ[P,Q,R]$, where
\[
P = 1-24\sum_{n\geq 1} \sigma_1(n)q^n
\]
will also play a r\^{o}le below in our discussion of the Heisenberg algebra.

\medskip
While elements of $M$ and $M'$ define holomorphic functions on the upper half plane, from Serre's perspective it is their Fourier coefficients --- and congruences among them --- that give rise to the $p$-adic theory.\
Observe that if $f \in M'$, then there exist an integer $N$ such that $Nf$ has an integral Fourier expansions. It follows that if $f$ has a Fourier expansion $f = \sum_{n\geq 0} a_nq^n$ for $a_n \in \QQ$, then for every prime, the sequence of $p$-adic absolute values $(\abs{a_n})_{n\geq 0}$ is bounded from above. Therefore, we can define a lift of the $p$-adic norm on $\QQ$ to $M'$, called the \emph{sup-norm}, as follows:
\[
 \abs{f} \df \sup_{n\geq 0} \abs{a_n}.
\]
\begin{rmk}
    Since we are working with a \emph{discrete} valuation, the supremum above could be replaced by a maximum, but if one wants to work for example over an algebraically closed field, then one should use the supremum as above.
\end{rmk}

The spaces $M$ and $M'$ are not complete with respect to this norm. To illustrate this point for $M$, recall that this ring contains all of the Eisenstein series
\[
E_k(q) = 1-\frac{2k}{B_k}\sum_{n\geq 1} \sigma_{k-1}(n)q^n,
\]
of weights $k\geq 4$, where $B_k$ is a Bernoulli number.\ Kummer's congruences for the Bernoulli numbers (see \cite{S1} for a statement of these congruences) have the following interpretation: let $(k_i)_{i\geq 0}$ be a sequence of integers that go to infinity in the Archimedean topology, but such that
\[
  k_{i+1} \equiv k_i \pmod{(p-1)p^{i}},
\]
and where $k_i \not \equiv 0 \pmod{p-1}$ for all $i$. Then the $p$-adic limit $k = \lim_{i}k_i$ exists, and the sequence $(E_{k_i})_{i\geq 0}$ converges with respect to the sup-norm to the following $p$-adic Eisenstein series:
\[
E_k^*(q) = 1-\frac{2k}{B_k^*}\sum_{n\geq 1} \sigma_{k-1}^*(n)q^n,
\]
where
\[
\sigma_{k-1}^*(n) = \sum_{\substack{d\mid n\\ p\nmid d}} d^{k-1}
\]
and $B_k^*$ is related to a special value of the $p$-adic Riemann zeta function in the same way that $B_k$ is related to the special values of the complex Riemann zeta function (see e.g. \cite{S1} for the precise relationship). 

\medskip
In this way, the Kummer congruences give examples of sequences in $M$ that are convergent, but such that the limit is not contained in $M$. For example, if the limit of the weights $k$ is a positive even integer, then the series $E_k^*$ is an Eisenstein series of weight $k$ on $\Gamma_0(p)$.

\medskip
In light of this, Serre introduces the completion $M_p$ of $M$ with respect to the sup-norm, and refers to it as the space of $p$-adic modular forms.\ His motivation was to give a new construction of the $p$-adic Riemann zeta function, and more generally Kubota-Leopoldt $p$-adic $L$-functions, by interpolating Eisenstein series as above.\ While not obvious, he also showed that $P \in M_p$ and thus $M'\subseteq M_p$.\ (Actually, Serre showed this for odd primes, and Katz explained how to deduce this when $p=2.$)

\medskip
Serre also introduced the space $X$ of $p$-adic weights as the space of continuous characters of $\ZZ_p^\times$.\ By taking a limit of $(\ZZ/p^r\ZZ)^\times$, one finds that \[\ZZ_p^\times \cong (1+p\ZZ_p) \times \ZZ/(p-1)\ZZ,\]
and thus one has an identification
\[
X\cong \begin{cases}
    \ZZ_p\times \ZZ/(p-1)\ZZ & p \textrm{ is odd},\\
    \ZZ_2 & p=2.
\end{cases}
\]
The integer weights are the diagonally embedded copy of $\ZZ$ in $X$, corresponding to the characters $z \mapsto z^k$.\ Note that $X$ is isomorphic to the profinite group $\varprojlim_r (\ZZ/p^r\ZZ)^{\times}$.\ As in the complex theory, there is a notion of even weights $2X$ and odd weights $1 +2X$.

\medskip
Serre showed that if $(f_i)_{i\geq 0}$ is a sequence of modular forms with weights $k_i$ satisfying congruences as above, then they have a limit $k$ in weight space.\ If the sequence $(f_i)$ converges to a $p$-adic modular form $f \in M_p$, then the $p$-adic weight of $f$ is defined to be $k$. 

\medskip
Spaces of $p$-adic modular forms have been important in recent decades for studying congruences among Fourier coefficients of modular forms, giving rise to constructions of new types of $L$-functions.\ They have also been critical for studying $p$-adic properties of Hecke algebras acting on spaces of modular forms.\ Many classical forms are $p$-adic modular forms.\ For example, Serre showed that every form on $\Gamma_0(p)$ is an element of $M_p$.\ He also observed that certain meromorphic forms, such as $j^{-1}$ when $p=2$, are (holomorphic) $p$-adic modular forms.

\section{$p$-adic vertex algebras}\label{SpVA}
 Throughout this Section $V$ is a $p$-adic Banach space.\ This will serve as our  state-space.

\subsection{$p$-adic fields on $V$}\label{SSpfield}
 Recall that a linear operator on $V$ is \textit{continuous} if, and only if, it is \textit{bounded}.\
Analogously to (\ref{dist}) we will be working with formal power series 
\begin{eqnarray*}
a(z)=\sum_{n\in\mathbf{Z}} a(n)z^{-n-1}\in End_{\mathbf{Q}_p}(V)[[z, z^{-1}]].
\end{eqnarray*}

The basic definition is as follows.\  We consider  the following two conditions, which should hold for all states $b\in V$:
\begin{eqnarray}\label{pfddef}
&&\ (i)\ \mbox{there is}\ M\in\mathbf{R}_+\ \mbox{(depending on $a(z)$)\ such that}\ |a(n)b|\leq M|b|\ \mbox{for all $n$},\\
&&(ii)\ \lim_{n\rightarrow\infty} a(n)b=0\notag
\end{eqnarray}

We say that $a(z)$ is a $p$-adic field if it satisfies (\ref{pfddef})(i), (ii).\ Set
\begin{eqnarray*}
\mathcal{F}(V):=\{p\mbox{-adic fields $a(z)$ on $V$}\}.
\end{eqnarray*}

We make several comments regarding (\ref{pfddef}).\ Part (i) is a minor modification of the corresponding definition in \cite{FM1}.\ It is 
equivalent to the statement  that the operator norms $|a(n)|$ are uniformly bounded by $M$.\ In particular
each $a(n)\in B(V)$.\ Part (ii) is the $p$-adic analog of the property $a(n)b=0\ (n\gg0)$ for algebraic vertex operators.\
It says that $\sum_{n\geq 0} a(n)b$ converges.\
 The next result is basic, cf.\  \cite{FM1}, Proposition 4.6.\

\begin{thm}\label{thmFVB} $\mathcal{F}(V)$ is a $p$-adic Banach space when equipped with the sup norm, i.e., $|a(z)|:=\sup_n|a(n)|$. 
\end{thm}
\begin{proof} It is easy to see that $\mathcal{F}(V)$ is a linear subspace of $End(V)[[z, z^{-1}]]$.\ We will show that it is closed.\
Let $(a_i(z))$ be a Cauchy sequence in $\mathcal{F}(V)$.\ Then $\forall \epsilon>0,\ \exists N>0$ such that for all  $i, j>N$ and all $n$ the inequality
$|a_i(n)-a_j(n)|<\epsilon$ holds.\ In particular $|a_i(n)|<\epsilon$.\ This says that for each $n$, 
$(a_i(n))$ is a Cauchy sequence in $End(V)$.\ It therefore has a limit, call it $a(n)\in End(V)$.

\medskip
Let $a(z):=\sum_n a(n)z^{-n-1}$.\ We have $|a_i(z)-a(z)|=\sup_n |a_i(n)-a(n)|=|(a_i(n)-a_j(n))+(a_j(n)-a(n))|<\epsilon$ for large enough $j$.\
Thus $a_i(z)\rightarrow a(z)$.\ It remains to show that $a(z)\in\mathcal{F}(V)$.\ For all $n$ and all $b$ we have $|a(n)b|=|(a(n)b-a_i(n)b)+a_i(n)b|\leq\sup\{|a(n)b-a_i(n)b|, |a_i(n)b| \}
=\sup\{|a(n)-a_i(n)|, M_i \}|b|\leq (\sup\{\epsilon, M_i\})|b|$ where we choose $i$ large enough so that $|a(n)-a_i(n)|<\epsilon$ and where $|a_i(n)b|\leq M_i|b|$
for all $n$ and all $b$.\ This proves (\ref{pfddef})(i).\ Similarly, to prove (\ref{pfddef})(ii), fix $b\neq0$ in $V$ and any $\epsilon>0$.\ Then $|a(n)-a_i(n)|<\epsilon/|b|$ for all 
$n$ and all large enough
$i$.\ Then $|a(n)b|=|(a(n)-a_i(n))b+a_i(n)b|\leq\sup\{\epsilon, |a_i(n)b|\}$.\ But $|a_i(n)b|<\epsilon$ for large enough $n$.\ Thus for large enough
$n$ we have $|a(n)b)<\epsilon$, as required.
\end{proof}

\subsection{$p$-adic $Y$-maps} A $p$-adic $Y$-map is a \textit{continuous} linear map
$$Y: V\rightarrow \mathcal{F}(V), a\mapsto Y(a, z)=\sum_n a(n)z^{-n-1}.$$
This makes sense after Theorem \ref{thmFVB}.\ Naturally, we call $Y(a, z)$ the $p$-adic field or $p$-adic vertex operator associated to $a$.\
The map $Y$ itself is the $p$-adic state-field correspondence.

\medskip
When we have such a $Y$-map it will be convenient to adapt the notation of (\ref{pfddef})(i) as follows:
\begin{eqnarray}\label{pfddef1}
|a(n)b|\leq|a||b|\ \mbox{for all $n$}.
\end{eqnarray}
This amounts to replacing $M$ in (\ref{pfddef})(i) by $|a|$.\ We shall always assume this for our $p$-adic vertex operators, although it is a slight departure from \cite{FM1}, \cite{FM2} where we used the condition $|a(n)b|\leq M|a||b|$ for the sake of generality because this was all that was required in proofs.\ 
(\ref{pfddef1}) is also consistent with Banach algebra theory as explained in the Example at the end of the next Subsection.

\subsection{$p$-adic vertex algebras} \label{SSpVA} A $p$-adic  vertex algebra is a triple $(V, Y, \mathbf{1})$ satisfying the following axioms:
\begin{eqnarray*}
&& (a)\ (\mbox{$p$-adic state-field correspondence})\ Y: V\rightarrow \mathcal{F}(V)\ \mbox{is a $p$-adic $Y$-map},\\
&& (b)\ (\mbox{Creativity})\ Y(a, z)\mathbf{1} = a+ O(z), \\
&& (c)\  |\mathbf{1}|\leq 1, \\
&& (d)\ (\mbox{$p$-adic Jacobi identity) for all}\ a, b, c\in V\ \mbox{and}\  r, s, t\in \mathbf{Z}\ \mbox{we have}\\
&&\ \ \ \ \ \ \ \ \ \ \ \ \ \ \sum_{i=0}^{\infty} {r\choose i}(a(t+i)b)(r+s-i)c = \\
&&\sum_{i=0}^{\infty} (-1)^i{t\choose i}\left\{a(r+t-i)b(s+i)c- b(s+t-i)a(r+i)c  \right\}.
\end{eqnarray*}

The examples we know all satisfy the condition $|\mathbf{1}|=1$.

\medskip
Let us consider the $p$-adic Jacobi identity more carefully.\ It is identical to the standard algebraic Jacobi identity presented in Subsection \ref{SSVA} but the meanings diverge.\
In the $p$-adic setting above the sums are not necessarily finite as they are in the algebraic setting.\ It is proved in \cite{FM1}, Lemma 4.9 that each of the three sums in the
$p$-adic Jacobi identity converge $p$-adically so that the identity itself makes sense.

\medskip\noindent
Example.\ Banach algebras.\
Let $V$ be a commutative $p$-adic Banach algebra.\
By definition,  multiplication in $V$ satisfies $|ab|\leq|a||b|$.\ Thus multiplication by $a$ is continuous and $V$ becomes a $p$-adic vertex algebra,
albeit of a rather trivial type, for which $Y(a, z):=a(-1)$ is the constant series consisting of the multiplication by $a$ map (i.e.,
$a(-1)b:=ab$).\ We leave it to the reader to contemplate the relation between the Jacobi identity and the commutativity of $V$
(cf.\ \cite{FM1}, Section 5).


\subsection{The $p$-adic state-field correspondence} In this Subsection we look more closely at a $p$-adic vertex algebra
$(V, Y, \mathbf{1})$ and especially the $Y$-map.\ The next Theorem clarifies, extends, and summarizes arguments in \cite{FM1}, Proposition 4.14 and Section 5.
\begin{thm} Suppose that $(V, Y, \mathbf{1})$ is a $p$-adic vertex algebra and let $W\subseteq \mathcal{F}(V)$ denote the image of $Y$.\ Then the following hold:
\begin{eqnarray*}
&&(a)\  W\ \mbox{is a closed subspace of}\ \mathcal{F}(V), \\
&&(b)\ W\ \mbox{is a $p$-adic vertex algebra with respect to residue products of fields}, \\
&&(c)\ Y: V\stackrel{\cong}{\longrightarrow} W\ \mbox{is a topological isomorphism of $p$-adic vertex algebras}.
\end{eqnarray*}
\end{thm}
\begin{proof} First one easily shows as a consequence of the creativity axiom that $|a|\leq |Y(a, z)|$.\ To prove (a),
suppose that $(Y(a_i, z))$ is a Cauchy sequence in $W$.\ From the previous inequality one deduces that $(a_i)$ is a Cauchy sequence in $V$.\
Therefore it has a limit, say $a\in V$, and by the continuity of $Y$ we find that $\lim_i Y(a_i, z)=Y(a, z)\in W$.\ This shows that $W$ is complete, whence it is closed
by a general property of metric spaces,
and (a) is proved.\ We skip any discussion of part (b) and refer the reader instead to \cite{FM1}, Section 5, where the definition and properties of  residue products
of $p$-adic fields are explained, including a proof that (\ref{pfddef1}) holds in $W$.

\medskip
Finally, since $Y: V\longrightarrow W$ is a continuous surjection of $p$-adic Banach spaces it is an open map by Banach's open mapping theorem\footnote{We thank Fernando Pe\~{n}a V\'{a}zquez for pointing this out (personal communication).}  which continues to hold in the $p$-adic setting \cite{BGR}.\
$Y$ is injective thanks to the creativity property, so $Y$ is bijective and therefore a topological isomorphism and part (c) is proved.
\end{proof}

\section{$p$-adic vertex operator algebras}\label{SpVOA}
So far we have limited our discussion to vertex algebras and their $p$-adic variants.\ In this Section we consider analogs of the additional axioms (a)-(d) listed in Subsection
\ref{SSVOA} that define a vertex operator algebra.\
We begin by stating our axioms.
\begin{dfn}\label{defpVOA} A $p$-adic vertex operator algebra is a quadruple $(V, Y, \mathbf{1}, \omega)$ satisfying the following.
\begin{eqnarray*}
&&(i)\ (V, Y, \mathbf{1})\ \mbox{is a $p$-adic vertex algebra},\\
&&(ii)\  \omega\in V\ \mbox{has vertex operator}\  Y(\omega, z)=\sum_n L(n)z^{-n-2}\ \mbox{with modes that satisfy}\\
&&\mbox{relations (a), (b) of Subsection \ref{SSVOA}},\\
&&(iii)\ \mbox{If $U$ is the $\mathbf{Q}_p$-span of the eigenstates for $L(0)$ belonging to eigenvalues}\\
&&\mbox{in $\mathbf{Z}$ then $U=\oplus U_{\ell}$
satisfies (c) and (d) of Subsection \ref{SSVOA}},\\
&&(iv)\ \mbox{$U$ is dense in $V$.}
\end{eqnarray*}
\end{dfn}

This amounts to a sharpening of the definition in \cite{FM1}.\ We make some additional comments.\ First, if the Virasoro modes close on the Virasoro relations with central charge $c$, then naturally we say that $V$ has central charge $c$.\ 

\medskip
Next, we have seen that all modes of all states in $V$ are bounded, hence continuous operators on
$V$ and so we can ask for a description of their spectra.\ This may vary considerably and in particular there may be an \textit{empty} point spectrum, i.e., no nonzero eigenstates, or there may  be nonintegral eigenvalues.\  Axiom (iii)
puts nontrivial restrictions on the point spectrum 
of $L(0)$.

\medskip
In the following let $V=(V,Y, \mathbf{1}, \omega)$ be a $p$-adic vertex operator algebra
and let $U$ be as above.\ The next Theorem is a good example of transfering standard VOA arguments to the $p$-adic setting.
\begin{thm}\label{thmUdense} $U=(U, Y, \mathbf{1}, \omega)$ is an algebraic vertex operator subalgebra of $V$ of central charge $c$.
\end{thm}
\begin{proof}
To be clear, the vertex operators 
for $a\in U$ are just 
the \textit{restrictions} $\Res^V_U Y(a, z)$.

\medskip
First note that by the creativity axiom for $p$-adic
vertex operators we have $L(0)\mathbf{1}=0$.\
So $\mathbf{1}\in U_0$.\ Similarly, $L(-2)\mathbf{1}=\omega$, so that $L(0)\omega=L(0)L(-2)\mathbf{1}=([L(0), L(-2)]+L(-2)L(0))\mathbf{1}=2L(-2)\mathbf{1}=2\omega$, where we used the Virasoro relations for the penultimate inequality.\ Thus $\omega\in U_2$.

\medskip
Suppose that $a\in U_{\ell}, b\in U_m$ and consider any mode
$a(n)$ of $a$.\ We claim that
$a(n)b\in U_{\ell+m-n-1}$.
This is a standard VOA computation that still works in the $p$-adic setting, cf.\ \cite{FM1}, Lemma 6.6.\ 
In particular, this shows that $U$ is closed under all $n^{th}$ products, i.e., $Y(a, z)b\in U[[z, z^{-1}]]$.\ More is true:\ since the grading on $U$ is bounded below by hypothesis then $Y(a, z)b$ has
a \textit{finite} singular part so the restriction of
$Y(a, z)$ to $U$ is a vertex operator in the algebraic sense.

\medskip
Finally, these vertex operators satisfy the (algebraic) Jacobi identity because they are restrictions of $p$-adic vertex operators that satisfy the $p$-adic Jacobi identity.\ This completes the proof of the Theorem.
\end{proof}
\section{Construction of some $p$-adic vertex algebras}\label{Sconstruct}
   Theorem \ref{thmUdense} says that every $p$-adic vertex operator algebra is the completion of an algebraic VOA.\ In this Section we present a technique to construct $p$-adic VOAs based on the observation that, in the notation of Theorem \ref{thmUdense}, the restriction of the norm on
   $V$ to $U$ makes $U$ in to a normed space as well as an algebraic VOA.\ Typically, $U$ will not be complete with respect to this norm.\ This will only happen if $U=V$.
\subsection{Normed algebraic vertex algebras}
In this Subsection we let $(V, Y, \mathbf{1})$  be an (algebraic) vertex algebra over $\mathbf{Q}_p$.\ We also assume that $V$ carries a $p$-adic norm  
$|\bullet|$.\ In this 
situation we use the notation
$\mathcal{F}(V)'$ to denote the $p$-adic fields $a(z)$ on $V$  that satisfy (\ref{pfddef})(i), (ii) in Section 5.2.\ This definition makes sense for any  normed space $V$ and does
\textit{not} require that $V$ is complete.\ 
$\mathcal{F}(V)'$ is a normed space with the usual sup norm which may not be complete.

\medskip
A normed algebraic vertex algebra over $\mathbf{Q}_p$  is a quadruple $(V, Y, \mathbf{1}, | \bullet|)$ satisfying the following axioms:
\begin{eqnarray*}
&&\ (i)\ |\mathbf{1}|\leq 1, \\
&& (ii)\ Y(a, z)\in\mathcal{F}(V)', \\
&&\ (iii)\ Y:V\rightarrow \mathcal{F}(V)'\ \mbox{is continuous}.
\end{eqnarray*}

In \cite{FM1} we called these \textit{algebraic vertex algebras with a compatible norm}.\

\medskip
In (ii), $Y(a, z)=\sum_na(n)z^{-n-1}$ is the algebraic vertex operator attached to the state $a$ in  $(V, Y, \mathbf{1})$.\ Thus
condition (\ref{pfddef})(ii) is automatically satisfied and  the content of (ii) is that the modes $a(n)$ have norms that are uniformly bounded by $|a|$.

\medskip
We can realize a $p$-adic vertex algebra by completing
$(V, Y, \mathbf{1}, | \bullet|)$.\ More precisely
\begin{thm}\label{thmcomp} Suppose that $(V, Y, \mathbf{1}, | \bullet|)$ is a normed algebraic vertex algebra, and let $V'$ be the completion of $V$.\ Then $V'$ is the underlying Banach space
of a $p$-adic vertex algebra.
\end{thm}
\begin{proof}The general idea should be evident.\ For a Cauchy sequence $(a_i)$ in $V$ let $a:=\lim_i a_i\in V'$.\ Since $Y$ is continuous then
$Y'(a, z):=\lim_i Y(a_i, z)$.\ We must show that $Y'(a, z)\in\mathcal{F}(V')$ so that $Y'$ is the state-field correspondence for $V'$.\ And we must verify the axioms of Subsection \ref{SSpVA}.\
For details, cf.\ \cite{FM1}, Proposition 7.2.
\end{proof}

This argument shows that the algebraic vertex operators $Y(a_i, z)$ extend by continuity to $V'$ and  $Y(a_i, z)=\mbox{Res}_V Y'(a_i, z)$.\ So we can think of $V$ as a dense vertex subalgebra of $V'$.

\medskip
One can achieve the hypotheses of Theorem \ref{thmcomp} as follows:\ Let $(V, Y, \mathbf{1}, \omega)$ be an algebraic VOA over $\mathbf{Q}_p$.\ Recall from 
Subsection \ref{SSVOA} that $V=\oplus_{\ell\in\mathbf{Z}} V_{\ell}$ has a direct sum decomposition defined by the eigenspaces of $L(0)$.\
By a \textit{Virasoro compatible $\mathbf{Z}_p$-lattice} of $V$ we mean  a $\mathbf{Z}_p$-submodule $U\subseteq V$  satisfying 
\begin{eqnarray}
&&(a)\ U=\oplus_{\ell} U_{\ell}, U_{\ell}\ \mbox{is a free $\mathbf{Z}_p$-module, and}\ V_{\ell}=\mathbf{Q}_p\otimes U_{\ell},\label{Ulattice}\\
&&(b)\ U\ \mbox{is closed with respect to all products}.\notag
\end{eqnarray}
We do \textit{not} insist that $\omega\in U$.

\medskip
 The easiest way to realize this is to start with a VOA
 $(U, Y, \mathbf{1}, \omega)$ defined over $\mathbf{Z}_p$, or even $\mathbf{Z}$,  and simply extend scalars by taking $V:=\mathbf{Q}_p\otimes U$.\ With this approach 
 the Virasoro vector $\omega$ belongs to $U$ whereas this is not required (and does not always hold in practice).
 
 \medskip
 Be that is it may, in the context of (\ref{Ulattice}) choose a $\mathbf{Z}_p$-basis for each $U_{\ell}$ and let
 $\{e_j\}$ denote the  resulting basis of $U$.\ Each state $a\in V$ is a finite sum $a=\sum_j t_je_j\ (t_j\in\mathbf{Q}_p)$ and we define
 $|a|:=\sup_j |t_j|$.\ 
 
 \medskip
 We have (\cite{FM1}, Proposition 7.4):
 \begin{thm}\label{thmnorm}The norm defined above is compatible with $(V, Y, \mathbf{1})$.\ In particular the completion $V'$ of $V$ defines a $p$-adic vertex algebra. $\hfill\Box$
 \end{thm}

 We will use the following alternate description of $V'$ later (cf.\ the proof of \cite{FM1}, Proposition 7.3.)
 \begin{lem}\label{lemalt}
Let the notation be as before.\ Then
$$ V'=\left\{\sum_{\ell} v_{\ell}\mid v_{\ell}\in V_{\ell}, v_{\ell}\rightarrow 0\right\}.$$ $\hfill\Box$
\end{lem}
 
\subsection{The $p$-adic Heisenberg vertex algebra}\label{SSU}
As we have explained, any algebraic VOA $(V, Y, \mathbf{1}, \omega)$ defined over $\mathbf{Q}_p$ and having a
Virasoro compatible $\mathbf{Z}_p$-lattice  gives rise to a $p$-adic vertex algebra by the process of completion.\ There are a number of VOAs  that are known to have a $\mathbf{Z}$-base,
so that they produce $p$-adic vertex algebras in this way.\ We will discuss just one example in detail.
We will look at this example more closely in later sections especially with regard to its connections with $p$-adic modular forms.\  Here we introduce the basic construction.\ To be clear, we are considering the algebraic Heisenberg
VOA of central charge $c=1$ defined over $\mathbf{Q}_p$,\ which we denote by $S_{alg}$.

\medskip
$S_{alg}$  is generated by a single state $h$ of weight $1$ and its vertex operator $Y(h, z)=\sum_n h(n)z^{-n-1}$ has modes that close on the Heisenberg Lie algebra,
that is they satisfy the canonical commutator relations 
\begin{eqnarray}\label{CCR}
[h(m), h(n)]= h(m+n)+m\delta_{m, -n} \mathbf{1}.\ 
\end{eqnarray}
The totality of states 
\begin{eqnarray}\label{Hbase}
\{h(-n_1)...h(-n_k)\mathbf{1}\mid n_1\geq...\geq n_k\geq 1\}
\end{eqnarray}
is a base of $S_{alg}$, indeed one sees  that the $\mathbf{Z}_p$-span $U$ of this base is a Virasoro compatible $\mathbf{Z}_p$-lattice in $S_{alg}$.\ Note that $\omega=\tfrac{1}{2}h(-1)^2\mathbf{1}$,
 so that $\omega$ only belongs to $U$ when $p$ is odd.

\medskip
States in $S_{alg}$ are $\mathbf{Q}_p$-linear combinations 
$ \sum_I a_I h^I$ where $I:=\{n_1, \hdots, n_k\}$ is a multiset  of decreasing positive integers and $h^I:=h(-n_1)\hdots h(-n_k)\mathbf{1}$.\ Then set
$$ \left|\sum_I a_Ih^I \right| :=\sup_I |a_I |.$$
This is  the compatible norm described in Theorem \ref{thmnorm}, so that Theorem \ref{thmcomp} guarantees that the completion is a $p$-adic vertex algebra.\ This is our $p$-adic Heisenberg
vertex algebra.\ We denote it by $S_1$.
 \section{Some categories}
 In this Section we consider some of the objects we have introduced from a more categorical perspective.\ We consider four concrete categories of relevance.
 \begin{eqnarray*}
   && \mathbf{A} = \mbox{category of $p$-adic vertex algebras},\\
    && \mathbf{B}= \mbox{category of finite-dimensional, unital, commutative, $p$-adic Banach algebras},\\
    && \mathbf{C}=\mbox{category of $p$-adic vertex operator algebras}, \\
    && \mathbf{D}=\mbox{category of $p$-adically normed algebraic vertex operator algebras}.
 \end{eqnarray*}

 Morphisms in $\mathbf{A}$ are \textit{continuous linear maps} $U\stackrel{f}{\longrightarrow}V$ such that $f$ preserves vacuum states as well as all products, i.e., $f(a(n)b)=f(a)(n)f(b)$.
 $\mathbf{C}$ is a subcategory of $\mathbf{A}$ such that morphisms also preserve Virasoro states.\ 
 $\mathbf{B}$ is a well-known category and we just remind the reader that morphisms are continuous maps.\
 Finally, morphisms $U\stackrel{f}{\longrightarrow}V$ in $\mathbf{D}$ are algebraic morphisms of the algebraic vertex operator algebras that are also continuous in the norm topologies of $U$ and $V$.

  \subsection{An equivalence of categories}
  The main result is
  \begin{thm}\label{thmequiv}
There is an equivalence of categories 
$\mathbf{C}\stackrel{\sim}{\longrightarrow}\mathbf{D}$.
  \end{thm}

To prove this we have to be more precise in our approach.\ For example, when dealing with completions of normed algebraic vertex algebra $V$ as in Theorem \ref{thmcomp},
we should really refer to a \textit{completion} as a map $V\stackrel{\iota}{\longrightarrow} V'$ such 
that $V'$ is a $p$-adic vertex algebra, $\iota(V)$ is dense in $V'$ and $\iota$ preserves all products in
$V$.\ The existence of such a completion in this sense is proved in \textit{loc.\ cit.}\ With this in mind, we first prove
  \begin{lem}\label{lemextension} Let $T\stackrel{f}{\longrightarrow} U$ be a morphism in $\mathbf{D}$, and let
$T\stackrel{i}{\longrightarrow} V$, $U\stackrel{j}{\longrightarrow} W$ be completions of $T$ and $U$ respectively.\ Then there is a \textit{unique} morphism
$V\stackrel{\hat{f}}{\longrightarrow} W$ in $\mathbf{C}$ that makes the following diagram commute.
\begin{eqnarray*}
\xymatrix{ & V\ar[rr]^{\hat{f}}&& W\\
&\\
& T\ar[rr]_f\ar[uu]^{i}&& U\ar[uu]_j
}
\end{eqnarray*}
\end{lem}
\begin{proof} By a standard result (\cite{BGR}, Subsection 1.1.7, Proposition 6), just at the level of normed spaces we know that there is a \textit{unique, continuous}  $\hat{f}$ that makes the same diagram commute.\
So it remains to show that this $\hat{f}$ is a morphism in $\mathbf{C}$, i.e., that it preserves all products in $V$.

\medskip
Because
$i(a(n)b)=i(a)(n)i(b)$ for states $a, b\in T$ and similarly for $j$ we have
$\hat{f}((i(a)(n)i(b))=   \hat{f}(i(a(n)b))=jf(a(n)b)) =j(f(a)(n)f(b))=(jf(a))(n)(jf(b))=(\hat{f}i(a))(n)(\hat{f}i(b))$.\ This proves the desired result for the
action of $\hat{f}$ on $i(T)$.\
In general, let $a:=\lim_i a_i,b:=\lim_j b_i$ with $a_i, b_j\in T$.\ Then 
$a(n)b=\lim_i\lim_ja_i(n)b_j$ and using the continuity of $\hat{f}$ we see that 
$$\hat{f}(a(n)b)=\lim_i\lim_j \hat{f}(a_i(n)b_j)= \lim_i\lim_j (\hat{f}(a_i))(n)(\hat{f}(b_j))=(\hat{f}(a))(n)(\hat{f}(b)),$$
as required.
\end{proof}

For each object $T$ in $\mathbf{D}$ we let $T\longrightarrow G(T)$ be a completion so that $G(T)$ is an object in $\mathbf{C}$.\ Moreover if
$T\stackrel{f}{\longrightarrow}U$ is a morphism in
$\mathbf{D}$ then
by Lemma \ref{lemextension} there is a commuting diagram
\begin{eqnarray*}
\xymatrix{ & G(T)\ar[rr]^{\hat{f}={G(f)}}&& G(U)\\
&\\
& T\ar[rr]_f\ar[uu]&& U\ar[uu]
}
\end{eqnarray*}
where the top row is an arrow in $\mathbf{C}$ which defines $G$.\ It is straightforward to check that 
$G:\mathbf{D}\longrightarrow\mathbf{C}$ is a
covariant functor.

\medskip
The functor $F:\mathbf{C}\longrightarrow\mathbf{D}$ is defined by restriction.\ Thus if $V\stackrel{f}{\longrightarrow}W$ is a morphism in $\mathbf{C}$ then
\begin{eqnarray*}
\xymatrix{ & V\ar[rr]^f\ar[dd]&& W\ar[dd]\\
&\\
& F(V)\ar[rr]_{res f=F(f)}&& F(W)
}
\end{eqnarray*}
where $F(V)$ is the span of the integral eigenstates of $L(0)$ in its action on $V$. We must show that the bottom row is a morphism in $\mathbf{D}$.\ That $F(V), F(W)$ are objects in
$\mathbf{D}$ follows from Theorem \ref{thmUdense}, and $F(f)$ preserves products $\mathbf{1}, \omega$ and all products because $f$ does.\ Then we check that $F$ is also a functor.

\medskip
Next we show that there is a natural transformation
$\alpha: G\circ F\sim I_{\mathbf{C}}$.\ If $V$ is an object in
$\mathbf{C}$ then $\alpha_V$ is defined by 
\begin{eqnarray*}
\xymatrix{ & V  && G(F(V))\ar[ll]_{\alpha_V}   \\
&\\
&&F(V)\ar[uur]\ar[uul]
}
\end{eqnarray*}
The two diagonal arrows are completions of $F(V)$ and
$\alpha_V$ is the unique morphism in $\mathbf{C}$ that makes the diagram commute (Lemma \ref{lemextension}).\ Now consider the following diagram where $T=F(V), U=F(W)$.
\begin{eqnarray*}
\xymatrix{&& V\ar[rrrr]^f &&&& W\\
&\\
&&& T\ar[rr]^{F(f)=res f}\ar[uul]\ar[ddl]&&U\ar[uur]\ar[ddr] \\
&\\
&& G(T)\ar[rrrr]_{(G\circ F)(f)}\ar[uuuu]^{\alpha_V} &&&& G(U)\ar[uuuu]_{\alpha_W}
}
\end{eqnarray*}
The two inner triangles commute by the definition of
$\alpha$, the inner quadrilaterals commute for functorial  reasons, and therefore the outer square commutes.\ This says that $\alpha$ is indeed a natural transformation, as asserted.

\medskip
Finally, by  a completely analogous procedure we can show that $F\circ G \sim I_{\mathbf{D}}$.\ We skip the details.\ This will then prove Theorem \ref{thmequiv}.
\subsection{Category $\mathbf{B}$}
We have already seen in the Example at the end of Subsection \ref{SSpVA} that a unital, commutative $p$-adic Banach algebra is a $p$-adic vertex algebra for which each vertex algebra $Y(a, z)=a(-1)$ is a constant. Indeed, the category of unital, commutative $p$-adic Banach algebras is a full subcategory of the category $\mathbf{A}$.

\medskip
We are interested here in such Banach algebras that are also finite-dimensional.\ We have
\begin{lem} $\mathbf{B}\subset \mathbf{C}$ is a full subcategory.\ It consists of the objects in $\mathbf{C}$ having Virasoro vector $\omega$ equal to $0$.
\end{lem}
\begin{proof} Let $V$ be an object in
$\mathbf{C}$ having $\omega=0$.\ Then $L(0)$ is the zero operator, i.e., $V=\ker L(0)$.\ By hypothesis
this has finite dimension, so $\dim V$ is finite.\ Suppose that $a, b\in V$.\ We have seen that for all integers $n$ we have $L(0)a(n)b=(-n-1)a(n)b$.\ So if $n\neq-1$ then $a(n)b=0$.\
It follows that the vertex operator
$Y(a, z)$ for $V$ is a constant $a(-1)$.\ Now one checks using the Jacobi identity that, when equipped with the product $a(-1)b$, $V$ is an associative, commutative algebra.\ Moreover by
(\ref{pfddef1}) we have $|a(-1)b|\leq|a||b|$.\ Thus
$V$ is a finite-dimensional, commutative $p$-adic Banach algebra.\ It has an identity element $\mathbf{1}$.\ Per the comments at the beginning of this Subsection every such $p$-adic Banach algebra is an object in $\mathbf{C}$ when regarded as having $\omega=0$, and the Lemma follows.
\end{proof}
  \section{$p$-adic modular forms and the $p$-adic Heisenberg vertex algebra, I}\label{SHI}
  Following the Introduction, we shall explain in this Subsection how $p$-adic modular forms are connected to the $p$-adic Heisenberg vertex algebra $S_1$.\

  \medskip
  Let $(V, Y, \mathbf{1}, \omega)$ be an algebraic VOA defined over $\mathbf{Q}_p$ having central charge $c$. For a homogeneous state $a\in V_{\ell}$
  we define $o(a):=a(\ell-1)$.\ This is the \textit{zero-mode} of $a$, so-called because one knows that
  $o(a): V_m\rightarrow V_m$ for all $m$ has weight $0$ as an operator.
 
\subsection{$1$-point functions}
  After the previous paragraph the following definition
  makes sense.
  \begin{eqnarray}\label{defF}
F(a) := \Tr_V o(a)q^{L(0)-c/24} := q^{-c/24} \sum_{\ell} \Tr_{V_{\ell}} o(a)q^{\ell}.
\end{eqnarray}
This is the $1$-point function for $a$.\ As it stands it is nothing but a formal $q$-series.
 We extend $F$ to a map on $V$ by linearity
so that we obtain 
\begin{eqnarray*}
F: V\longrightarrow q^{-c/24}\mathbf{Q}_p[q^{-1}][[q]].
\end{eqnarray*}

 \textit{Modular invariance in  VOA theory} by and large refers to the investigation of the modular properties of $1$-point functions.\ The reader who is uninured with VOA theory should be surprised that these $1$-point functions have anything to do with  modular forms.
We illustrate with some relevant examples in which
$V=S_{alg}$.

\medskip
The most basic example of a $1$-point function is the \textit{graded character} of $V$, aka the \textit{partition function}, defined to be
  \begin{eqnarray}\label{defF1}
F(\mathbf{1}) := \Tr_V q^{L(0)-c/24} := q^{-c/24} \sum_{\ell} \dim V_{\ell} q^{\ell},
\end{eqnarray}
In the case of $S_{alg}$ it is well-known that
 \begin{eqnarray*}
F(\mathbf{1}) =q^{-1/24}\sum_{n=0}^{\infty} p(n)q^n
:= \eta(q)^{-1}=q^{-1/24}\prod_{n=0}^{\infty} (1-q^n)^{-1},
\end{eqnarray*}
the inverse of Dedekind's eta-function.

\medskip
It is convenient to renormalize the $1$-point function $F$ by multiplying by $\eta$.\ Thus we define
\begin{eqnarray*}
Z: S_{alg} \rightarrow \mathbf{Q}_p[[q]], \ \ a\mapsto \eta(q) F(a).
\end{eqnarray*}

Recall from Section \ref{Spadic} the algebra 
$\mathbf{Q}_p[E_2, E_4, E_6]$
of quasimodular forms, denoted by $Q$ in the Introduction.\ Then we have \cite{MT}, \cite{MT2}, \cite{DMN}.
\begin{thm}\label{thmZsurj} $Z$ defines a linear surjection
\begin{eqnarray*}
Z: S_{alg}\longrightarrow Q.
\end{eqnarray*}
$\hfill\Box$
\end{thm}

Having previously described $S_{alg}$ and $Q$ as $p$-adically normed linear spaces the next result is crucial \cite{FM1}:
\begin{thm} $Z$ is a continuous map. 
$\hfill\Box$
\end{thm}

By a basic property we can uniquely extend $Z$ to  a continuous map between the completions.\ Since the completion
of $S_{alg}$ is the $p$-adic Heisenebrg vertex algebra  $S_1$ and the completion of $Q$ is the space of $p$-adic modular forms $M_p$, we obtain

\begin{thm}\label{thmcommdiag}There is a unique continuous extension $\widehat{Z}$ of $Z$ so that the following diagram commutes
\begin{center}
\xymatrix{ S_1\ar[rr]^{\widehat{Z}} && M_p\\
S_{alg}\ar@{->>}[rr]_Z\ar[u]&& Q\ar[u]
}
\end{center}
$\hfill\Box$
\end{thm}
In this way we have defined $1$-point functions for the $p$-adic Heisenberg vertex algebra $S_1$.\ Of course 
if $a\in S_1$ then $\widehat{Z}(a)$ is generally not a trace function like (\ref{defF}) but rather a limit of (normalized) trace functions.

\subsection{The square bracket VOA}\label{SSsqbracket}In the next few Subsections we will explain how to refine the discussion of the previous Subsection by taking \textit{weights} into account.\
In particular, this entails a discussion of how to make $Z$ into a graded map.\ This is a well-known feature of VOA theory, and this Subsection presents the basic facts
about this.\ For further details, see e.g., \cite{MT}, \cite{Z}.

\medskip
Let $(V, Y, \mathbf{1}, \omega)$ be an algebraic VOA of central charge $c$ defined over $\mathbf{Q}_p$.\ (In practice we only consider the Heisenberg VOA $S_{alg}$.)\
Then there is a second VOA $(V, Y[\ ], \mathbf{1}, \widetilde{\omega})$ called the \textit{square bracket} VOA.\ It has the same state space and vacuum vector, however the new
Virasoro vector is defined by $\widetilde{\omega}:= \omega-\tfrac{c}{24}\mathbf{1}$.\ Most importantly, the $Y$-map is defined as follows.\ For $a\in V_{\ell}$ we set
\begin{eqnarray*}
&&Y[a, z] := e^{\ell z}Y(a, e^z-1) :=\sum_{n\in\mathbf{Z}} a[n]z^{-n-1},\ \ Y[\widetilde{\omega}, z]:=\sum_{n\in\mathbf{Z}} L[n]z^{-n-2}
\end{eqnarray*}
and then extend $Y[\ ]$ linearly to $V$.\ Evidently, one may express the square bracket modes $a[n]$ in terms of the $a(n)$ modes (\cite{DLM}, \cite{Z}), but this will generally be unenlightening except for the
following two Virasoro modes  of particular relevance to us:
\begin{eqnarray}\label{Virmodes}
&&L[-1] = L(0)+L(-1),\ \ L[0] = L(0)+\sum_{n=1}^{\infty} \tfrac{(-1)^{n+1}}{n(n+1)}L(n).
\end{eqnarray}

As a matter of fact there is an isomorphism of algebraic VOAs (\cite{H}, \cite{FBZ}, \cite{Z}) $(V, Y[\ ], \mathbf{1}, \widetilde{\omega})\stackrel{\cong}{\longrightarrow} (V, Y, \mathbf{1}, \omega)$,
although here this construction  is mainly important for the following reason:\ as the underlying state space of the square bracket VOA, $V$ acquires a second conformal decomposition which we write as
\begin{eqnarray*}
&& V= \oplus_{\ell\in\mathbf{Z}} V_{[\ell]}, \ \  V_{[\ell]}:=\{a\in V\mid L[0]a = \ell a\}.
\end{eqnarray*}
This is called the square bracket grading on $V$.\ Applying this machinery to the algebraic Heisenberg VOA, the basic fact  is this
(\cite{DMN}, \cite{MT}, \cite{MT2}):
\begin{thm}\label{thmgdd} If $S_{alg}$ is equipped with its square bracket grading and $Q$ its natural $\mathbf{Z}$-grading then the normalized $1$-point function
$Z$ of Theorem \ref{thmZsurj}  is a graded map.\ That is, if $a\in V_{[\ell]}$ then $Z(a)$ is a quasimodular form of weight $\ell$. $\hfill\Box$
\end{thm}

\subsection{$X$-weights in the $p$-adic Heisenberg vertex algebra}\label{SSX}
 We have already explained Serre's theorem on weights of $p$-adic modular forms in Section \ref{Spadic}.\ We can now use this to
 show that certain limit states in the $p$-adic Heisenberg vertex algebra $S_1$ can similarly acquire a weight in $X$.\ To be precise, we have
 \begin{thm}\label{thmXwt}  Suppose that $(a_i)$ is a Cauchy sequence of states in $S_{alg}$ such that $a_i\in V_{[\ell_i]}$.\ 
 Then $(\ell_i)$ is a Cauchy sequence with a limit
 $\ell\in X$.\ If $a:=\lim_i a_i\in S_1$ we say that $a$ has $X$-weight $\ell$.\
 Then $\widehat{Z}$ is a graded map  in the sense that it preserves $X$-weights of states and
 $\widehat{Z}(a)=\lim_i Z(a_i)$ is a $p$-adic modular form of weight $\ell$.
 \end{thm}
 \begin{proof}
 As a consequence of Theorem \ref{thmgdd}  $(Z(a_i))$ is a Cauchy sequence
 in the normed space $Q$ of quasimodular forms  and $Z(a_i)$ has weight $\ell_i$.\  Now Serre's theorem says that $(\ell_i)$
 has a limit $\ell\in X$.\ This is the weight of the $p$-adic modular form $\lim_i Z(a_i)=\widehat{Z}(a)$,
 the latter equality following from Theorem \ref{thmcommdiag}.
 \end{proof}
  
  The fact that the limit states  $a\in S_1$ occuring in Theorem \ref{thmXwt} have an $X$-weight in the sense defined
  above is a statement that makes sense  independently of the theory of $p$-adic modular forms.\
  Nevertheless we needed Serre's theorem to
  prove the statement - we know of no other proof.\ 
  
  \medskip
  We also point out that this notion of $X$-weight is an extension of the conformal weight of states in $S_{alg}$.\ That is,
  if $a\in (S_{alg})_{[\ell]}$ has square bracket conformal weight $\ell$ then $a$ has $X$-weight $\ell$.
  \section{Virasoro action and the Hamiltonian spectra}
  We have seen that some states in the $p$-adic Heisenberg algebra $S_1$ acquire a weight in the weight space $X$, but we have not explicitly considered analogs of the Virasoro axioms (a)-(d) of Subsection \ref{SSVOA}
 as they may, or may not, pertain to $S_1$.\ We take up this topic in this Section.\ In particular we prove
 (Corollary \ref{corpVOAHeis}) that $S_1$ is indeed a $p$-adic vertex operator algebra.
  
  \subsection{Virasoro action}\label{SSViract}
 We remind the reader
  that if $a\in S_{alg}$ is an algebraic state then all of  its modes $a(n)$ act continuously on $S_{alg}$ and hence they have a unique extension to 
  a continuous map on the completion $S_1$, also denoted by $a(n)$.
   
  \begin{lem}\label{lemS1Vir} As operators on $S_1$ the Virasoro modes $L(n)$ close on the Virasoro algebra with $c=1$.
  \end{lem}
  \begin{proof} This follows directly from continuity.\ If $(a_i)$ is a Cauchy sequence of states in $S_{alg}$ with $a:=\lim_i a_i\in S_1$ then
  $[L(m), L(n)]a=\lim_i ([L(m), L(n)]a_i)=\lim_i \{ ((m-n)L(m+n)+\delta_{m, -n}\tfrac{m^3-m}{12}Id_{S_{alg}})a_i\}=((m-n)L(m+n)+\delta_{m, -n}\tfrac{m^3-m}{12}Id)a$,
  and the Lemma follows.
  \end{proof}

 This shows that Axiom (a) of Section \ref{SSVOA} holds, and Axiom (b) is similarly unchanged as one easily sees by a continuity argument similar to the previous proof.\ Axiom (c) cannot possibly hold in its entirety on account of the
uncountable dimension of $S_1$.\ In fact we have the following precise description of the spectrum of $L(0)$ as it acts on $S_1$.

\begin{thm}\label{thmL0} As an operator on $S_1$ the Hamiltonian $L(0)$ has a pure point spectrum equal to $\mathbf{Z}_{\geq0}$.\
The
span of the $L(0)$-eigenstates  is equal to the algebraic vertex subalgebra $S_{alg}$.
\end{thm}
\begin{proof}\ Concerning $\mbox{spec} L(0)$, what has to be shown is that if
$\lambda\in\mathbf{Q}_p$ then either
$\lambda\in\mathbf{Z}_{\geq0}$ or else 
$L(0)-\lambda I$ is an invertible operator on $S_1$.\ The second statement of the Theorem says, in effect, that upon passing from $S_{alg}$ to $S_1$ the operator $L(0)$ acquires \textit{no new eigenstates}.\
This is proved in \cite{FM1}, Proposition 7.3, so we concentrate on showing that $L(0)-\lambda I$ is invertible if 
$\lambda\notin\mathbf{Z}_{\geq0}$.

\medskip
Fix such a $\lambda$ and let $R$ be the range of $L(0)-\lambda I$.\ We claim that it suffices to show that $R$ is closed in $S_1$.\
Indeed,
$L(0)-\lambda I$ is injective and leaves each homogeneous subspace $(S_{alg})_m$ invariant.\ Because $(S_{alg})_m$ is finite-dimensional it lies in $R$, and therefore $S_{alg}\subseteq R$.\
So if $R$ is closed then we have $R=S_1$ because $S_{alg}$ is dense.\ This proves that $L(0)-\lambda I$ is surjective as well as injective and we are done.

\medskip
Now we prove that $R$ is complete and therefore closed.\ From Lemma \ref{lemalt} it is
 obvious that 
 $$ R=\left\{ \sum_m (m-\lambda)v_m \mid v_m\in (S_{alg})_m, v_m\rightarrow 0  \right\}.$$
 
 Let $(a^i)$ be a Cauchy sequence of states in $R$ with
$$a^i :=\sum_n (n-\lambda)v_n^i\ \ (v_n^i\in V_n).$$
Thus $\forall \epsilon>0, \exists N$ such that $\forall i, j>N$ we have $|a^i-a^j|<\epsilon$.

\medskip
Fix any $m\geq0$.\ We claim that $(v^i_m)_i$ is a Cauchy sequence in $(S_{alg})_m$.\ We have
$a^i-a^j=\sum_{n\geq 0} (n-\lambda)(v_n^i-v_n^j)$,
and hence
$\mid a^i-a^j\mid = \sup_n |n-\lambda||v_n^i-v_n^j|$
and $\forall \epsilon>0, \exists N$ such that $\forall i, j>N$ we have $|a^i-a^j|<\epsilon|m-\lambda|$.\ Noting that $m-\lambda\neq0$, then
 $\forall \epsilon>0, \exists N$ such that $\forall i, j>N$ we have
$$|v_m^i-v_m^j|  = \tfrac{1}{|m-\lambda|}  |m-\lambda||v_m^i-v_m^j| \leq\tfrac{1}{|m-\lambda|}|a^i-a^j|<\epsilon.$$
This proves the Claim.\ Since $V_m$ is finite-dimensional it is closed in $S_1$, hence the limit $u_m$ of the Cauchy sequence $(v_m^i)_i$
lies in $V_m$.\
Consider $a:= \sum_{m\geq 0}(m-\lambda)u_m\in R$.\ We have
$a^i-a = \sum_{m\geq 0} (m-\lambda)(v_m^i-u_m)$
and
$$ \lim_i |a^i-a| = \lim_i\sup_m |m-\lambda||v_m^i-u_m|\leq\max\{1, |\lambda|\}\lim_i\sup_m|v_m^i-u_m|=0,$$
the last equality holding because $u_m=\lim_i v_m^i$ for every $m$.\ So $\lim_i a^i=a\in R$ and the proof of the Theorem is finished.
\end{proof}

\begin{cor}\label{corpVOAHeis}
$S_1$ is a $p$-adic vertex operator algebra.
\end{cor}
\begin{proof}
  Recall the definition of $p$-adic vertex operator algebra  (Definition \ref{defpVOA} in Section \ref{SpVOA}.) We already know that $S_1$ is a $p$-adic vertex algebra.\ Next, Lemma \ref{lemS1Vir} and the ensuing discussion shows that part (ii) of Definition \ref{defpVOA} is satisfied.\ Let $U$ be the 
  $\mathbf{Q}_p$-span of the eigenstates for 
  $L(0)$.\ By Theorem \ref{thmL0}, $U$ is the underlying linear space for $S_{alg}$ and in particular $U$ is dense in $S_1$.\ Thus parts (iii) and  (iv) of Definition \ref{defpVOA} are also satisfied and the Corollary is proved.
\end{proof}
\subsection{The unbounded operator $L[0]$} In $S_{alg}$ the Hamiltonians $L(0)$ and
$L[0]$ have identical point spectra.\ This is a consequence of the  isomorphism between
$(V, Y, \mathbf{1}, \omega)$ and $(V, Y[\ ], \mathbf{1}, \widetilde{\omega})$ (cf.\ the discussion in Subsection \ref{SSsqbracket}).\ Thus one might
think that $L[0]$ also satisfies the conclusions of Theorem \ref{thmL0}.\ We shall explain in this Subsection that this is far from true.

\medskip
First we reconsider Theorem \ref{thmXwt}.\ We will adopt the notation used there.\ That discussion did not utilize the equalities
$L[0]a_i=\ell_i a_i$, but we use them now.\ Consider the following instructive, but \textit{fake} argument with $a:=\lim_i a_i\in S_1$:
\begin{eqnarray}\label{fake}
L[0]a=L[0] \lim a_i = \lim_i L[0]a_i = \lim_i \ell_i a_i =(\lim_i \ell_i)(\lim_i a_i) = \lambda a
\end{eqnarray}
 where $\lambda:= \lim_i \ell_i$ considered as a limit
in $\mathbf{Z}_p$ rather than $X$.\ 

\medskip
This argument suggests that $L[0]$ has plenty of eigenstates in $S_1$ belonging to eigenvalues $\mathbf{Z}_p$ that are not necessarily rational integers.\
In fact this is  correct in a sense and we shall give a precise formulation in Theorem \ref{thmL0spec} below.\ However (\ref{fake}) does not establish this
because
it tacitly assumes that $L[0]$ is  a continuous operator on $S_1$ so that it passes through the limit to get the second equality.\ The formula for
$L[0]$ in (\ref{Virmodes}) suggests that this may not be true and that we must treat $L[0]$ as an \textit{unbounded operator}.

\medskip
To handle this situation we introduce some additional normed spaces following \cite{FM1}, Section 9 and \cite{FM2}, Section 2.6.\ We use the notation
$\sum a_I h^I$ for states in $S_{alg}$ as described in Subsection \ref{SSU}.\ Additionally, if $I=\{n_1, \hdots, n_k\}$ we set $|I|=\sum n_i$.\ 
For a real number $R\geq 1$, we introduce a new ultrametric norm on  $S_{alg}$ as follows:
\begin{eqnarray}\label{SRnormdef}
\left|\sum_I a_I h^I \right|_R := \sup_I |a_I| R^{|I|}.
\end{eqnarray}
Let $S_R$ denote the completion of  $S_{alg}$  equipped with the norm $|\ |_R$.\ $S_R$ is in fact a $p$-adic Banach space \cite{FM1} that coincides with our $p$-adic Heisenberg vertex algebra when $R=1$ (whence the notation).\ We have in general $S_{alg}\subseteq S_{R_1} \subseteq S_{R_2}\subseteq S_1$ whenever $R_1\geq R_2\geq 1$
and 
\begin{eqnarray*}
S_R=\left\{\sum_I a_Ih^I \mid \lim_I |a_I| R^{|I|} = 0 \right\}.
\end{eqnarray*}

\medskip
To get to the point of all of this, one can show \cite{FM2}, Theorem 3.4 that if $R> p^{1/p}$ then $L[0]$ leaves $S_R$ invariant and acts continuously on it.\
One thus considers $L[0]$ as an unbounded operator on $S_1$ according to the maps
\begin{eqnarray*}
S_R\stackrel{L[0]}{\longrightarrow} S_R \stackrel{\iota}{\longrightarrow} S_1
\end{eqnarray*}
where $\iota$ is the set-theoretic inclusion of the underlying linear spaces.\ $\iota$ is an open but  not continuous map.\ For general background dealing with similar scenarios see e.g.,
\cite{DR}, Chapter 6.

\medskip
We must now look for Cauchy sequences $(a_i)$ in $S_{alg}$  with a limit $a\in S_{R}$.\ Returning to
the situation of
(\ref{fake}), the argument is correct when we regard the first limit as one in $S_R$ because $L[0]$ is continuous there and we may move it past the limit sign.\ We then have \cite{FM2}, Theorem 4.4:
\begin{thm}\label{thmL0spec} Suppose that $p^{1/p}<R <p^{1/(p-1)}$.\ Then for any $\lambda\in\mathbf{Z}_p$ there is a state $a\in S_R$ such that $L[0]a=\lambda a$.\ $\hfill\Box$
\end{thm}

We have already explained the reason for the lower bound on $R$.\ The upper bound  arises from the need to construct suitable Cauchy sequences
$(a_i)$, which gets more burdensome with increasing $R_1$.\ As in Subsection \ref{SSX} and  especially Theorem \ref{thmXwt}, this process is intimately related to the connections with $p$-adic modular forms.\ The next Section is devoted to this subject.

\medskip
Finally, we mention a more specialized result, cf.\ \cite{FM2}, Corollary 6.9:
\begin{thm}\label{thmp=2} Suppose that $p=2$ and that $2^{1/2}<R<2^{3/4}$.\ Then the weight $0$ eigenspace for $L[0]$ in its action on $S_R$ is infinite-dimensional.
$\hfill\Box$
\end{thm}

The restriction to $p=2$ is made for technical reasons.\ The result is likely to be true in much greater generality.\

\section{$p$-adic modular forms and the $p$-adic Heisenberg vertex algebra, II}\label{SHII}
In this Section we enlarge on the proofs of Theorems \ref{thmL0spec} and \ref{thmp=2}.\ We use the ideas and notation of Section \ref{Spadic}

\medskip
The connection between eigenstates for $L[0]$ and $p$-adic modular forms comes about from 
the commuting diagram of Theorem \ref{thmcommdiag} which we have set-up so that the  horizontal maps are weight-preserving.\
For $\widehat{Z}$ this means that  if
$a\in S_1$ and $\widehat{Z}(a)\in M_p$ is a $p$-adic modular form of weight $k\in X$ then 
$a$ also has weight $k$.

\medskip
 Because $Z$ is surjective then the image of $\widehat{Z}$ contains all
 quasimodular forms and it is natural to ask whether 
 $\widehat{Z}$ is itself a surjection.\ This question,
 which has been a motivating question for us all along, remains open.\ Theorems \ref{thmL0spec} and \ref{thmp=2}
 represent incremental steps towards its resolution.\
 
 \medskip
 What we must do is find a sequence of states $(a_i)$
  with $a_i\in (S_{alg})_{[k_i]}$  such that the sequence is Cauchy not just in $S_1$ but in $S_R$ with $R>p^{1/p}$ to ensure the continuity of $L[0]$.\ Furthermore we want $k_i\rightarrow k$
 in $X$.\ 
 To be clear, this is only possible if $k$ is an even weight.\ This is so because nonzero $p$-adic modular forms have a weight in \textit{even} weight space $2X$.

 \begin{thm}\label{thmEk} For any $0\neq k\in 2X$ there is a Cauchy sequence $(a_i)$ in $S_R$
 ($R$ as in Theorem \ref{thmL0spec}) such that $\lim_i Z(a_i)=E_k^*$.\ In particular
 $E_k^*\in im\ \widehat{Z}$.
 $\hfill\Box$
\end{thm}
 
 We remind the reader that $E_k^*$ is the $p$-adic Eisenstein series of weight $k$ discussed in Section \ref{Spadic}.\ Theorem \ref{thmEk} is stated as Theorem 1.1 in  \cite{FM2} and proved in Sections 3 and 4 (loc.\ cit.)\  Together with $Z(\mathbf{1}) =1$, which has weight $0$, it handles all even weights in $X$.\ In particular by (\ref{fake}) we get Theorem \ref{thmL0spec} for all 
 $\lambda \in 2\mathbf{Z}_p$.

 \medskip
 To deal with odd weights, one uses (\ref{Virmodes}) to see that $L[-1]$ is a bounded operator on $S_R$.\ Then a standard calculation using $[L[0], L[-1]]=L[-1]$  shows that $L[-1]$ maps an $L[0]$-eigenstate with eigenvalue $\lambda$ to another eigenstate with eigenvalue $1+\lambda$.\ Furthermore we can choose such eigenstates to be nonzero, so that there are nonzero eigenstates for $L[0]$ having $X$-weight equal to any given 
 $\lambda\in X$.\ Then one deduces Theorem \ref{thmL0spec}.

 \medskip
 As for Theorem \ref{thmp=2}, we prove much more in \cite{FM2}, Theorem 1.1(b).\ Here is a watered down version that still implies Theorem \ref{thmp=2}.
 \begin{thm} If $p=2$ the subspace of 
 $\im \widehat{Z}$ spanned by the 
 $2$-adic modular forms of weight $0$ contains 
 $\mathbf{Q}_2(j^{-1})$. $\hfill\Box$
 \end{thm}

The main point is to establish that $j^{-1}$ lies in the image of $\widehat{Z}$.\ This is done by a messy explicit calculation in \cite{FM2}, Section 5.
\bibliographystyle{amsplain}

\end{document}